\pdfoutput=1
\RequirePackage{ifpdf}
\ifpdf 
\documentclass[pdftex]{sigma}
\else
\documentclass{sigma}
\fi

\newcommand{\del}{\partial}

 \def\mmat #1,#2,#3,#4,{\text{\small\arraycolsep=3pt $
\begin{pmatrix}#1&#2\\#3&#4\end{pmatrix}$}}

\usepackage[all]{xy}
\usepackage{tabularx,mdwtab, mathrsfs}

\usepackage{stmaryrd}

\usepackage{bbm}

\newcommand {\supplus}{\mathop{{\supset}\llap{\raise
0.5pt\hbox{\normalfont\small+}\hskip 0.5pt}}}

\newcommand {\subplus}{\mathop{{\subset}\llap{\raise
0.5pt\hbox{\normalfont\small+}\hskip 0.5pt}}}

\newcommand {\divby} {\lower 0.15ex \hbox{\,\vdots\,}}
\newcommand {\Cee} {{\mathbb C}}

\newcommand {\Zee} {{\mathbb Z}}

\newcommand {\fb} {{\mathfrak{b}}}
\newcommand {\fc} {{\mathfrak{c}}}

\newcommand {\fder} {{\mathfrak{der}}} %

\newcommand {\fg} {{\mathfrak{g}}} %
\newcommand {\fgl} {{\mathfrak{gl}}} %
\newcommand {\fh} {{\mathfrak{h}}}
\newcommand {\fhei} {{\mathfrak{hei}}}
\newcommand {\fm} {{\mathfrak{m}}}
\newcommand {\fn} {{\mathfrak{n}}}

\newcommand {\fo} {{\mathfrak{o}}}
\newcommand {\fosp} {{\mathfrak{osp}}}
\newcommand {\fp} {{\mathfrak{p}}} %
\newcommand {\fpe} {{\mathfrak{pe}}} %

\newcommand {\fpo} {{\mathfrak{po}}}

\newcommand {\fpsl} {{\mathfrak{psl}}}

\newcommand {\fq} {{\mathfrak{q}}}

\newcommand {\fs} {{\mathfrak{s}}}

\newcommand {\fsl} {{\mathfrak{sl}}}

\newcommand {\fspe} {{\mathfrak{spe}}}

\newcommand {\fsq} {{\mathfrak{sq}}}

\newcommand {\fsvect} {{\mathfrak{svect}}}

\newcommand {\ft} {{\mathfrak{t}}}

\newcommand {\fv} {{\mathfrak{v}}} %
\newcommand {\fvect} {{\mathfrak{vect}}} %

\let\cal\relax
\newcommand {\cal} {\mathcal}

\newcommand {\cC} {{\cal C}}

\def \opname#1#2%
 {\expandafter\newcommand \csname #1\endcsname {\mathop{\mathrm{#2}}\nolimits }}

\newcommand{\rmname}[1]
 {\expandafter\newcommand \csname #1\endcsname {\mathop{\mathrm{#1}}\nolimits}}

\newcommand{\rmnameii}[2]
 {\expandafter\newcommand \csname #1\endcsname {\mathop{\text{\rm #2}}\nolimits}}

\rmname{act}
\rmname{Ad}
\rmname{Add}
\rmname{ad}
\rmname{Alt}
\rmname{alt}
\rmname{Ann}
\rmname{antidiag}
\rmname{Ber}
\rmname{ber}
\rmname{Bil}
\rmname{Br}
\rmname{But}
\rmname{card}
\rmname{ch}
\rmnameii{Char}{char}
\rmname{Ch}
\rmname{cem}
\rmname{cj}
\rmname{Cliff}
\rmname{cntr}
\rmname{Coker}
\rmname{codim}
\rmname{coind}
\rmname{Conn}
\rmname{conj}
\rmname{const}
\rmname{col}
\rmname{cork}
\rmname{Covect}
\rmname{cpr}
\rmname{diag}
\rmnameii{Div}{div}
\rmname{Def}
\rmname{Der}
\rmname{Diff}
\rmname{Dim}
\rmname{End}
\rmname{Ev}
\rmname{Even}
\rmname{Ext}
\rmname{Gal}
\rmname{gr}
\rmname{grad}
\rmname{Herm}
\rmname{Hom}
\rmname{HT}
\rmnameii{Ht}{ht}
\rmname{hwt}
\rmname{Id}
\rmname{id}
\rmname{ind}
\rmname{Ind}
\rmname{inv}
\rmname{Inv}
\rmname{Coind}
\rmname{Inf}
\rmname{irr}
\rmname{Ker}
\rmname{Le}
\rmname{Lie}
\rmname{lwt}
\rmname{Mat}
\rmname{mult}
\rmname{Mor}
\rmname{nm}
\rmname{Ob}
\rmname{Odd}
\rmname{ord}
\rmname{Osc}
\rmname{per}
\rmname{pet}
\rmname{Pf}
\rmname{Pic}
\rmname{pr}
\rmname{pro}
\rmname{Prime}
\rmname{Proj}
\rmname{prt}
\rmname{pt}
\rmname{ptr}
\rmname{Q}
\rmname{QE}
\rmname{QSt}
\rmname{qch}
\rmname{qet}
\rmname{qtr}
\rmname{rd}
\rmname{rk}
\rmname{row}
\rmname{Reg}
\rmname{Res}
\rmname{salt}
\rmname{Sch}
\rmname{sch}
\rmname{SBr}
\rmname{scalar}
\rmname{sdim}
\rmname{Ser}
\rmname{sign}
\rmname{SIGN}
\rmname{Smbl}
\rmname{spin}
\rmname{srk}
\rmname{ssym}
\rmname{sym}
\rmname{str}

\rmname{st}
\rmname{sgn}
\rmname{sq}
\rmname{symm}
\rmname{supp}
\rmname{Supp}
\rmname{St}
\rmname{Spec}
\rmname{Spm}
\rmname{tr}
\rmname{Vect}
\rmname{vpt}
\rmname{weyl}
\rmname{Weyl}
\rmname{Witt}

\opname{vvol} {{v\hspace{-0.1ex}o\hspace{-0.02ex}l\/}}
\opname{pnt} {\text{\normalfont pt}}
\opname{Span} {{Span}}
\opname{slim} {\overline{\lim}}
\opname{Vol} {{V\hspace{-0.55ex}o\hspace{-0.02ex}l\/}}
\opname{QVol} {{Q\hspace{-0.3ex}V\hspace{-0.55ex}o\hspace{-0.02ex}l\/}}
\opname{PoVol}{{P\hspace{-0.35ex}o\hspace{-0.25ex}V\hspace{-0.55ex}o\hspace{-0.02ex}l\/}}
\opname{BVol} {{B\hspace{-0.2ex}V\hspace{-0.55ex}o\hspace{-0.02ex}l\/}}
\opname{Par} {{P\hspace{-0.3ex}a\hspace{-0.05ex}r\/}}
\opname{Pty} {{P\hspace{-0.3ex}t\hspace{-0.05ex}y\/}}

\rmname{Wall}

\rmnameii {IM} {Im}
\rmnameii {RE} {Re}

\opname{Aut} {{A\hspace{-0.2ex}u\hspace{-0.1ex}t\/}}
\opname{GL} {{G\hspace{-0.3ex}L}}
\opname{SL} {{S\hspace{-0.3ex}L}}
\opname{Exp} {{E\hspace{-0.2ex}x\hspace{-0.1ex}p\/}}
\opname{GQ} {{G\hspace{-0.2ex}Q}}
\opname{OSp} {{O\hspace{-0.25ex}S\hspace{-0.15ex}p\/}}
\opname{Out} {{O\hspace{-0.25ex}u\hspace{-0.15ex}t\/}}
\opname{Spp} {{S\hspace{-0.2ex}p\/}}
\opname{SpO} {{S\hspace{-0.2ex}p\hspace{-0.02ex}O\/}}
\opname{Pe} {{P\hspace{-0.25ex}e\/}}
\opname{PGL} {{P\hspace{-0.25ex}G\hspace{-0.25ex}L\/}}
\opname{SPe} {{S\hspace{-0.25ex}P\hspace{-0.25ex}e\/}}
\opname{Spin} {{S\hspace{-0.25ex}p\hspace{-0.05ex}i\hspace{-0.1ex}n\/}}
\opname{Iso} {{I\hspace{-0.25ex}s\hspace{-0.1ex}o\/}}
\opname{Int} {{I\hspace{-0.25ex}n\hspace{-0.1ex}t\/}}
\opname{SSPe} {{S\hspace{-0.25ex}S\hspace{-0.15ex}P\hspace{-0.25ex}e\/}}
\opname{PeU} {{P\hspace{-0.25ex}e\hspace{-0.1ex}U\/}}
\opname{QU} {{Q\hspace{-0.15ex}U\/}}
\opname{U} {{U\/}}

\opname{cGQ} {{\cal G \hspace{-0.2em} Q \/}}
\opname{cSL} {{\cal S \hspace{-0.2em} L \/}}
\opname{cGL} {{\cal G \hspace{-0.2em} L \/}}\opname{cGr} {{\cal G \hspace{-0.2em} r \/}}
\opname{cOSp} {{\cal O \hspace{-0.2em} S \hspace{-0.3em} \it p\/}}
\opname{cPe} {{\cal P \hspace{-1.5pt} \it e\/}}
\opname{cVect} {{\cal V \hspace{-1.5pt} \it e\hspace{-0.1ex}c\hspace{-0.1ex}t\/}}
\opname{cVol} {{\cal V \hspace{-1.5pt} \it o\hspace{-0.1ex}l\/}}
\opname{cAut} {{\cal A \hspace{-0.2em} \it u\hspace{-0.1em}t\/}}
\opname{cCovect} {{\cal C \hspace{-1.5pt}
 \it o\hspace{-0.1ex}v\hspace{-0.1ex}e\hspace{-0.1ex}c\hspace{-0.1ex}t\/}}
\opname{CW} {{C\hspace{-0.15ex}W}}

\newcommand {\ev} {{\bar0}}

\opname {Ab} {{\sf Ab}}
\opname {Algs} {{\sf Algs}}
\opname {ASch} {{\sf Aff\;Sch}}
\opname {SSch} {{\sf SSch}}
\opname {Coh} {{\sf Coh}}\opname {Funct} {{\sf Funct}}
\opname {Gr} {{\sf Gr}}
\opname {Grf} {{{\sf Gr}_f}}
\opname {Man} {{\sf Man}}
\opname {MMan} {{\sf MMan}}
\opname {Mods} {{\sf Mods}}
\opname {SMods} {{\sf SMods}}
\opname {Rings} {{\sf Rings}}
\opname {Salgs} {{\sf Salgs}}
\opname {ScommSalgs} {{\sf ScommSalgs}}
\opname {LieAlgs} {{\sf LieAlgs}}
\opname {Sets} {{\sf Sets}}
\opname {SSMan} {{\sf SMan}}\opname {Sh} {{\sf Sh}}
\opname {Top} {{\sf Top}}
\opname {Vebun} {{\sf Vebun}}
\opname {Bun} {{\sf Bun}}

\newcommand {\bcdot} {\mathbin{\hbox{\raise.4ex\hbox{\bf.}}}} 

\numberwithin{equation}{section}

\newtheorem{Theorem}{Theorem}[section]

\newtheorem{Lemma}[Theorem]{Lemma}
\newtheorem{Proposition}[Theorem]{Proposition}

\begin{document}

\newcommand{\arXivNumber}{2503.03339}

\renewcommand{\PaperNumber}{047}

\FirstPageHeading

\ShortArticleName{Maximal Graded Solvable Subalgebras}

\ArticleName{Maximal Graded Solvable Subalgebras of Simple\\ Finite-Dimensional Vectorial Lie Superalgebras}

\Author{Irina SHCHEPOCHKINA}

\AuthorNameForHeading{I.~Shchepochkina}

\Address{Independent University of Moscow, 11 B.~Vlasievsky per., 119 002 Moscow, Russia}
\Email{\href{mailto:iparam53@gmail.com}{iparam53@gmail.com}}

\ArticleDates{Received March 06, 2025, in final form June 19, 2025; Published online June 24, 2025}

\Abstract{Here, in every simple finite-dimensional vectorial Lie superalgebra considered with the standard grading where every indeterminate is of degree~1, the maximal graded solvable subalgebras are classified over~$\mathbb{C}$.}


\Classification{17B20; 17B30}

\section{Introduction}
Hereafter, the ground field is $\Cee$, all spaces are finite-dimensional. For the notation of simple Lie superalgebras and related basics of linear superalgebra, see \cite{Lo}.

In \cite{Dy}, Dynkin solved the problem (now classical) ``describe maximal subalgebras of simple finite-dimensional Lie algebras''.
In \cite{ShM}, the superanalog of the problem Dynkin solved for subalgebras of \textit{any} type, was solved for the maximal \textit{non-simple irreducible}  subalgebras in \textit{linear} (realised by means of linear operators or matrices) Lie superalgebras $\fgl(p|q)$ and $\fsl(p|q)$, in their queer analogs $\fq(n)$ and $\fsq(n)$, in $\fpe(n)$ that preserves the non-degenerate odd bilinear form, in its supertraceless subalgebra $\fspe(n)$, in a~version $\fpe_\lambda(n)$ of $\fpe(n)$, and in the simple Lie superalgebra $\fosp(m|2n)$ that preserves the non-degenerate even bilinear form.

In this note, there are described and classified maximal \textit{graded} solvable Lie subsuperalgebras in the finite-dimensional simple \textit{vectorial} (i.e., realised by means of vector fields) Lie superalgebras of the four analogs of the Cartan series of simple Lie algebras. These ambients are considered with their standard grading where every indeterminate is of degree~1. The particular case of $\fvect(0|2)\simeq \fsl(1|2)$ was solved in~\cite{S2} in terms of supermatrices. Here, the answer is the union of five propositions which require different prerequisites in various cases and are proved differently.\looseness=-1

\section{Related results}
S.~Lie was the first to try to classify maximal subalgebras of simple vectorial Lie algebras
with polynomial coefficients. To classify all maximal subalgebras seems
to be a~wild (hopeless) problem: there are too many such subalgebras. Several researchers distinguished various classes of maximal subalgebras important in applications (e.g.,~\cite{Og,So}) or dealt with cases feasible as, for example, are the classifications of the maximal graded subalgebras in finite-dimensional simple vectorial Lie superalgebras of Cartan type over $\Cee$, see \cite{BLM}, and over algebraically closed fields of characteristic $p>3$, see~\cite{BLLM}.

In \cite{S2}, the maximal, \textit{not necessarily graded}, solvable Lie subsuperalgebras of $\fgl(p|q)$ and $\fsl(p|q)$ are described. These maximal solvable Lie subsuperalgebras are sometimes larger and sometimes smaller (see \cite[Theorem~3]{S2}) than the super analog of the \textit{Borel subalgebra} defined by Penkov as any Lie subsuperalgebra spanned by the maximal torus and root vectors corresponding to either all positive or all negative roots, see~\cite{P}; recall, for comparison, that the Borel subalgebras of simple Lie algebras are maximal solvable. The case of $\fh'(0|4)\simeq \fpsl(2|2)$ solved here is not considered in \cite{S2}, where the closely related cases of $\fgl(2|2)$ and $\fsl(2|2)$ are solved, yielding, however, totally distinct types of maximal \textit{non-graded} solvable subalgebras, see Proposition~\ref{P5} and an elucidation in its proof.

Kuznetsov \cite{Ku} classified maximal graded solvable subalgebras of finite codimension in simple vectorial Lie algebras over algebraically closed fields of characteristic $p>3$: a problem analogous to the one solved here for the graded solvable Lie algebra $\fs$ in the case where $\fs_{-1}=0$, see Section~\ref{3cases}\,(i).

\section{The ambients} There are four series of simple finite-dimensional vectorial Lie superalgebras; let us recall their description, see \cite{Lo} containing also details of history.

Let $\Lambda(n)$ or $\Lambda(\xi)$ designate the Grassmann algebra generated by $\xi:=(\xi_1, \dots, \xi_n)$ with the \textit{standard} grading ($\deg \xi_i=1$ for all $i$) and parity $p$ equal to the degree modulo 2. Here, the maximal \textit{graded}\, solvable Lie subsuperalgebras of simple Lie superalgebras of vector fields on the $0|n$-dimensional supermanifold $\cC^{0|n}$ are classified, the ambient superalgebras being considered with their $\Zee$-grading corresponding to the standard grading of $\Lambda(n)$, the algebra of functions on $\cC^{0|n}$. This grading, called \textit{Weisfeiler} grading, is associated with the \textit{Weisfeiler filtration} constructed by means of the maximal subalgebra consisting of the fields vanishing at the origin, see, e.g., \cite{BGLLS}. Note that the Lie superalgebra $\fsvect(0|n)$ of divergence-free vector fields has a~filtered deformation $\widetilde{\fsvect}(0|n)$ which can be endowed with $\Zee$-gradings, but these gradings are not Weisfeiler gradings. In \cite{BLM}, a~natural $\Zee/n\Zee$-grading of $\widetilde{\fsvect}(0|n)$, induced by the standard $\Zee$-grading of $\fsvect(0|n)$, is considered and the maximal subalgebras of $\widetilde{\fsvect}(0|n)$ are described for~$n$ even.\looseness=-1

Let $\fvect(0|n):=\fder\, \Lambda(n)$, also denoted $\fvect(\xi)$, be the Lie superalgebra of superderivations of $ \Lambda(n)$, its elements can be expressed as $\sum f_i\del_i$, where $f_i\in\Lambda(n)$ and $\del_i\xi_j:=\delta_{ij}$. The symbol $\fb(\xi)$
designates the subalgebra $\fb\subset \fvect(\xi)$ realized in coordinates $\xi$.

Let $\fsvect(0|n):=\{D\in\fvect(0|n)\mid \Div D=0\}$, where the \textit{divergence} is
\[
\Div \left(\sum f_i\del_i\right):=\sum (-1)^{p(f_i)}\del_i(f_i).
\]

Let $\fpo(0|n)$ be the Poisson Lie superalgebra realized on the superspace $\Lambda(n)$ generated for $n=2k+l$ by
\begin{gather}\label{indH}
\xi:=(\xi_1, \dots, \xi_k), \qquad \eta:=(\eta_1, \dots, \eta_k), \qquad \zeta:=(\zeta_1, \dots, \zeta_l),
\end{gather}
with the Poisson bracket
\[
\{f, g\}:=(-1)^{p(f)}\left(\sum_{1\leq i\leq k}\left(\frac{\del f}{\del \xi_i}\frac{\del g}{\del {\eta_i}} +\frac{\del f}{\del \eta_i}\frac{\del g}{\del {\xi_i}}\right)+\sum_{1\leq j\leq l}\frac{\del f}{\del \zeta_j}\frac{\del g}{\del {\zeta_j}}\right).
\]
Let $\fh(0|n):=\Span\{H_f\mid f\in\Lambda(n)\}$, where{\samepage
\[
H_f:=(-1)^{p(f)}\left(\sum_{1\leq i\leq k}\left(\frac{\del f}{\del \xi_i}\del_{\eta_i} +\frac{\del f}{\del \eta_i}\del_ {\xi_i}\right)+\sum_{1\leq j\leq l}\frac{\del f}{\del \zeta_j}\del_ {\zeta_j}\right).
\]
Clearly, $\fh(0|n)$ is the quotient of $\fpo(0|n)$ modulo center (spanned by constants). Although in this note we only need linear algebra and no geometry, recall, for the sake of completeness, that $\fh(0|n)$ preserves a~non-degenerate super-anti-symmetric closed differential 2-form $\omega$ on a~$(0|n)$-dimensional supermanifold, see~\cite{Lo}.

}

Hereafter, $\fg$ designates one of the Lie superalgebras $\fvect(0|n)$ (resp.\ $\fsvect(0|n)$ or $\fh(0|n)$) in their standard gradings induced by the standard $\Zee$-grading of $\Lambda(n)$ for $n>1$ (resp.\ $n>2$ or $n>3$) when these Lie superalgebras are simple, except $\fh(0|n)$.
Note that the description of maximal graded solvable subalgebras in $\fh(0|n)$ yields the description of same type of subalgebras in the simple derived algebra $\fh'(0|n)$; we describe the elements $H_f$ of $\fh(0|n)$ in terms of generating functions~$f$.

Let $\Xi=\xi_1\cdots\xi_n$, let $t$ be an extra indeterminate of parity $n\pmod 2$ and $\deg(t)=0$. The Lie superalgebras $\widetilde{\fsvect}(0|2n):=(1+t\Xi)\fsvect(0|2n)$ are isomorphic for any $t\neq 0$, so we assume $t=1$. The standard $\Zee$-grading of $\fsvect(0|n)$ induces a~$\Zee/n$-grading of $\widetilde{\fsvect}(0|n)$, the degrees being represented by integers from $-1$ to $n-2$. The $\Zee$- or $\Zee/n$-grading of the Lie superalgebra~$\fs$ is said to be \textit{compatible} (with parity) if $\fs_\ev=\bigoplus_{i\geq 0} \fs_{2i}$. Clearly,
\begin{gather*}
\widetilde{\fsvect}_i(0|n)=\begin{cases}\fsvect_i(0|n)&\text{for $n$ even},\\
\fsvect_i(0|n)\otimes \Cee[t]&\text{for $n$ odd},\end{cases} \qquad \text{for $i\neq -1$,}\\ \widetilde{\fsvect}_{-1}(0|n)=\Span\bigl(\tilde{\del}_1,\dots,\tilde{\del}_n\bigr), \qquad \text{where $\tilde{\del}_i:=(1+t\Xi)\del_i$}.
\end{gather*}
Here, we will not consider deformations with odd parameter.

\begin{Lemma}\label{L1} Let $\fs:=\bigoplus_{i\geq -1} \fs_i$ be a~compatible with parity $\Zee$- or $\Zee/n\Zee$-grading of a~Lie superalgebra $\fs$. Then, $\fs$ is solvable if and only if so is $\fs_0$.
\end{Lemma}

\begin{proof} In \cite[Proposition~1.3.3]{K}, based on the description of irreducible modules over solvable Lie superalgebras proved later, see~\cite{Sg}, it is proved that ``the Lie superalgebra $\fg$ is solvable if and only if its even component $\fg_\ev$ is solvable''. (Note that this statement, true over any field of characteristic not~2, does not hold over fields of characteristic 2; for examples where $\fg$ is simple while $\fg_\ev$ is solvable, see \cite{BGLLS, BGL}.) Since the grading of $\fs$ is compatible with parity, then $\fs_\ev$ contains a~nilpotent ideal $\fn:=\bigoplus_{i>0} \fs_{2i}$ and $\fs_\ev/\fn=\fs_0$. So, $\fs$~is solvable if and only if so is~$\fs_0$.
\end{proof}

\section{The three possible cases}\label{3cases}
Let $\fs$ be a~maximal graded solvable subalgebra of $\fg$, i.e., $\fs_i=\fs\cap\fg_i$ for all $i$. Let us consider the possible cases:
(i)~$\fs_{-1}=0$, (ii)~$\fs_{-1}=\fg_{-1}$, and (iii)~$0\neq V=\fs_{-1}\subsetneq\fg_{-1}$. Since each case is determined by its $(-1)$st component, we accordingly designate the maximal solvable subalgebra of $\fg$ as (i) $\fm\fs0^{\fg}$, (ii) $\fm\fs\fc^{\fg}$, (iii) $\fm\fs V^{\fg}$, where $\fm\fs$ means ``maximal solvable'' (subalgebra),
with~0, or complete, or equal to $V$ partial $(-1)$st component, respectively.

{\bf (i) $\boldsymbol{\fs_{-1}=0}$.} Obviously, if $\fs$ is maximal, then
\[
\text{$\fs_0$ is a~maximal solvable subalgebra in $\fg_0$ and $\fs_i=\fg_i$ for $i>0$}.
\]
Let $(\fm\fs0^{\fg})_0$ be a~maximal solvable
subalgebra in $\fg_0$; set
\begin{gather}\label{ms}
\fm\fs0^{\fg}:= (\fm\fs0^{\fg})_0\oplus\Bigl(\bigoplus_{i>0} \fg_i\Bigr).
\end{gather}

\begin{Proposition}\label{P1} Lie superalgebras $\fm\fs0^{\fvect(0|n)}$ and $\fm\fs0^{\fsvect(0|n)}$ for $n>2$, as well as $\fm\fs0^{\fh(0|n)}$ for $n>4$, are maximal solvable in $\fvect(0|n)$, $\fsvect(0|n)$ and $\fh(0|n)$, respectively. Moreover, the algebra isomorphic to $\fm\fs0^{\fsvect(0|2n)}$ is maximal solvable in $\widetilde{\fsvect}(0|2n)$.
\end{Proposition}

\begin{proof} From the classical results on maximal solvable subalgebras of $\fsl(n)$ and $\fo(n)$, we see that the bases of $\fs_0$ and $\fg_0$ are as follows (for a~reason the $\fs_0$ are chosen differently --- sometimes upper-triangular, sometimes lower-triangular --- see clarifications in Section~\ref{Defs}):

\begin{table}[!ht]\centering\renewcommand{\arraystretch}{1.2}
\footnotesize

\begin{tabular}{|l |l |l |l|}
\hline
$\fg$&$\fg_0$&basis of $\fg_0$&basis of $\fs_0$\\
\hline
$\fvect(0|n)=\fvect(\xi)$&$\fgl(n)$&$\xi_i\del_j$, where $1\leq i, j\leq n$&$\xi_i\del_j$ for $i\geq j$\\
\hline
$\fsvect(0|n)=\fsvect(\xi)$&$\fsl(n)$&$\xi_i\del_j$ for $i\neq j$ and $h_i:=\xi_i\del_i-\xi_{i+1}\del_{i+1}$&$\xi_i\del_j$ for $i> j$ and\\
 &&for $i\in\{1, \dots, n-1\}$&$h_i$ for $i\in\{1, \dots, n-1\}$\\
\hline
$\fh(0|2k)=\fh(\xi, \eta)$&$\fo(2k)$&$\xi_i\xi_j$ and $\eta_i\eta_j$ for $i> j$,&\\
&& $\xi_i\eta_j$ where
$1\leq i, j\leq k$&$\xi_i\xi_j$ for $i>j$,  $\xi_i\eta_j$ for $i\leq j$\\
\hline
$\fh(0|2k+1)=\fh(\xi, \eta, \zeta)$&$\fo(2k+1)$&$\xi_i\xi_j$ and $\eta_i\eta_j$ for $i> j$,\ $\xi_i\eta_j$\ where &$\xi_i\xi_j$ for $i>j$,  $\xi_i\eta_j$ for $i\leq j$ \\
 &&$1\leq i, j\leq k$ and $\xi_i\zeta$,  $\eta_i\zeta$ where $1\leq i\leq k$&and $\xi_i\zeta$ for $i\in\{1, \dots, k\}$\\
\hline
\end{tabular}
\caption{}\label{tab3}
\end{table}

Assume the contrary:
\begin{gather}\label{ft}
\text{let $\ft$ be a~solvable graded subsuperalgebra in $\fg$ containing $\fs$.}
\end{gather}
Since $\fs_0$ is a~maximal solvable subalgebra in $\fg_0$, then $\ft_+:=\bigoplus_{i\geq 0} \ft_i=\fs$ by Lemma~\ref{L1}.

Hence,
\begin{gather}\label{*}
[\ft_{-1}, \fg_1]\subset \fs_0.
\end{gather}

If $n>2$, then for $\fg=\fvect(0|n)$ and $\fsvect(0|n)$, the component $\fg_1$ contains the elements $v_1:=\xi_1\xi_2\del_3$ and $v_i:=\xi_1\xi_i\del_2$ for any $i>2$.
Clearly, the condition
\begin{gather}\label{test}
\text{$[\ft_{-1}, \fv_i]\subset \fs_0$ \qquad for all $i$}
\end{gather}
implies that $\ft_{-1}=0$, so $\ft=\fs$.

Let $\langle S\rangle:=\Span(S)$ designate the space spanned by the set $S$; let $\Lambda^i(\langle S\rangle )$ be the component of degree~$i$ in $\Lambda(\langle S\rangle)$.
For the subspaces $\langle \xi\rangle \otimes \Lambda^2(\langle \eta\rangle )$ and $\Lambda^3(\langle \eta\rangle )$ in the case of $\fg=\fh(0|2k)$ for $k\geq 3$ as well as for the subspaces $\langle \xi\rangle \otimes \Lambda^2(\langle \eta\rangle )$ and $\langle \zeta\rangle \otimes \Lambda^2(\langle \eta\rangle )$ in the case of $\fg=\fh(0|2k+1)$ for $k\geq 2$, the condition~\eqref{*} implies $\ft_{-1}=0$.
\end{proof}

\section{Definitions needed in cases (ii) and (iii) (see \cite{Shch})}\label{Defs} The graded Lie (super)algebra
$\fb=\bigoplus_{k \geqslant -d} \fb_k$
is said to be \textit{transitive} if for all~$k\geq 0$ we have
\begin{equation*}
\{x\in \fb_k \mid [x,\fb_{-}]=0\}=0, \qquad \text{where $\fb_{-}:=\oplus _{k <0}\,\fb_k$}.
\end{equation*}

The maximal transitive $\Zee$-graded Lie (super)algebra $\fg=\bigoplus_{k \geqslant -1}\ \fg_k$, whose non-positive part is the given $\fg_{-1}\oplus \fg_0$, is called the \textit{Cartan prolong}
--- the result of Cartan prolongation --- of the pair $(\fg_{-1},\fg_0)$ and is denoted by $(\fg_{-1},\fg_0)_*$.

Let $\fg := \fvect(0|n)$ or $\fsvect(0|n)$ or $\fh(0|n)$. Let, besides, $\fb_{-1} \subset \fg_{-1}$ be a non-zero subspace in $\fg_{-1}$, and $\fb_0 \subset \fg_0$ a subalgebra (not necessarily maximal) preserving this $\fb_{-1}$. Clearly, $\fb_{\leq 0}: = \fb_{-1} \oplus
\fb_0$ is a subalgebra in $\fg$. Define the \textit{prolong of
$\fb_{\leq 0}$ in $\fg$} as $\Zee$-graded subalgebra $\fb = \bigoplus_{k\geq -1}\fb_k$, where
$
\fb_k := \{D\in \fg_k \mid [D, \fb_{-1}] \subset \fb_{k-1}\ \text{for all} \ k \geq 1\}$.

If $\fb_{-1} = \fg_{-1}$, we get the
Cartan prolong
$(\fg_{-1}, \fb_0)_*$. We retain the analog of this notation for any non-zero subspace $\fb_{-1}\subset\fg_{-1}$. If $\fb_0$ is a maximal subalgebra of $\fg_0$ preserving $\fb_{-1}$, let us shrink the notation $(\fb_{-1}, \fb_0)_*$ to $(\fb_{-1})_*$.

\begin{Lemma}\label{L3} \quad\samepage
\begin{enumerate}\itemsep=0pt
\item[$1)$] The bracket in $\fg=\fvect(0|n)$ for $n> 2$ $($resp.\ $\fsvect(0|n)$ for $n>2$ or $\fh(0|n)$ for $n> 4)$, induces a~Lie superalgebra structure on the space $(\fb_{-1}, \fb_{0})_*$.

\item[$2)$] The Lie superalgebra $(\fb_{-1}, \fb_{0})_*$
is maximal among all graded Lie subsuperalgebras $\fb$ of $\fg$ with the given components $\fb_{-1}$ and $\fb_{0}$; the Lie superalgebra $(\fb_{-1})_*$ is maximal among all subalgebras of $\fg$ with given negative component.
\end{enumerate}
 \end{Lemma}

\begin{proof} Cf.~\cite{BLM}. Claim~1) is a~direct corollary of the Jacobi identity. Claim~2) is evident. \end{proof}

We would like to select the maximal solvable subalgebra $\fs_0$ of $\fg_0$ so it would
look as in Table~\ref{tab3}. Therefore, in case (iii), if $\fs_{-1}= V$, we'd like to numerate the basis $\fg_{-1}$ beginning with that of~$V$.

If $\fg$ is of series $\fvect$ or $\fsvect$, we assume that $V$ is spanned by $\del_1, \dots , \del_k$. But then the subalgebra of $\fg_0$ preserving $V$ contains the operators of the form $\xi_j\del_i$ for $i\le k$ and $j>k$, as commuting with $V$. Therefore, in these cases, it is convenient for us to select the \textit{lower}-triangular subalgebra in $\fg_0$ for $\fs_0$.

If $\fg$ is of series $\fh$, the element $\xi$ acts as $\del \eta$. Therefore, if $\xi_i$ for $i=1,\dots ,k$ lie in $V\cap V^\perp$ (hence, $\eta_i$ for $i=1,\dots ,k$ do not lie in~$V$), and there are also the $\xi_j$, $\eta_j$ with $i, j >k$ (wherever they lie), then $\xi_i\eta_j$ with $i\le k$ and $j>k$ must lie in the subalgebra of $\fg_0$ preserving $V$. Therefore, in this case, it is convenient for us to select the \textit{upper}-triangular subalgebra in $\fg_0$ for $\fs_0$.

{\bf (ii) $\boldsymbol{\fs_{-1}=\fg_{-1}}$.}

\begin{Proposition}\label{P2}\quad
\begin{enumerate}\itemsep=0pt
\item[$1)$] Let $\fg=\fvect(0|n)$ for $n> 2$ $($resp.\ $\fsvect(0|n)$ for $n>2$ or $\fh(0|n)$ for $n> 4)$, and let $\fs_0$ be the maximal solvable Lie subalgebra in $\fg_0$, and $\fm\fs\fc^\fg:=(\fg_{-1}, \fs_0)_*$.
The Lie superalgebra $\fm\fs\fc^\fg$ is maximal solvable in $\fg$.
\item[$2)$] There are no maximal solvable Lie subalgebras $\fs$ of $\fg=\widetilde \fsvect(0|2n)$ with $\fs_{-1}=\fg_{-1}$.
\end{enumerate}
\end{Proposition}

\begin{proof}
1)~is a~direct corollary of Lemma~\ref{L1}.
2)~Observe that the $(-1)$st component of $\widetilde{\fsvect}(0|2n)$ generates the whole algebra, so $\widetilde{\fsvect}(0|2n)$ has no maximal solvable subalgebra containing the whole $(-1)$st component, see \cite[Lemma~2.4]{BLM}.
\end{proof}

{\bf Explicit descriptions of $\boldsymbol{\fm\fs\fc^\fg}$.}
If $\fs_0$ is chosen as in Table~\ref{tab3}, then for a~basis of the degree-$(m-1)$ component $\fs_{m-1}^{\fvect(0|n)}$ of $\fm\fs\fc^{\fvect(0|n)}$ one can take monomials
\[
\xi_{i_1}\cdots\xi_{i_m}\del_j, \qquad \text{where $i_a\geq j$ for any $a\in\{1, \dots, m\}$}.
\]

Clearly,
\[
\fm\fs\fc^{\fsvect(0|n)}=\fm\fs\fc^{\fvect(0|n)}\cap\fsvect(0|n).
\]

For a basis of the $(m-1)$st component $\fm\fs\fc_{m-1}^{\fh(0|2k)}$, we can take monomials
\[
\xi_{i_1}\cdots\xi_{i_m}\eta_j, \  \text{where $i_a\leq j$ for any $a\in\{1, \dots, m\}$},  \qquad \text{and}\qquad
\xi_{i_1}\cdots\xi_{i_{m+1}} \  \text{for any $i_a$}.
\]

For a~basis of the $(m-1)$st component $\fm\fs\fc_{m-1}^{\fh(0|2k+1)}$, we can take the union of a basis of $\fs_{m-1}^{\fh(0|2k)}$ with the monomials
\[
\xi_{i_1}\cdots\xi_{i_m}\zeta\qquad \text{for any $i_a$}.
\]
Let $\fhei(\zeta):=\fhei(0|1)$ be the Heisenberg Lie superalgebra spanned by an odd $\zeta$ and even $z$ with the multiplication table
\[
[z, \fhei(\zeta)]=0,\qquad [\zeta, \zeta]=z.
\]
Let $\overline{\Lambda}(\xi)$ be the quotient of $\Lambda(\xi)$ modulo constants. Then,
\begin{gather}
\fm\fs\fc^{\fh(0|2k)}\simeq \overline{\Lambda}(\xi)\inplus\fm\fs\fc^{\fvect(0|k)}(\xi),\nonumber\\
\fm\fs\fc^{\fh(0|2k+1)}\simeq (\overline{\Lambda}(\xi)\otimes\fhei(\zeta))\inplus\fm\fs\fc^{\fvect(0|k)}(\xi).\label{(*)}
\end{gather}
The generating functions whose degree with respect to $\eta$ is $\leq 1$ span a~Lie subsuperalgebra $\fg$ in $\fh(0|2k)$ isomorphic to $\bar\Lambda(\xi)\inplus\fvect(\xi)$. Note that $\fg$ is precisely the degree-0 component of $\fh(\xi,\eta)$ in the non-standard grading $\deg_{ns} (f):=\deg (f)-1$, where $\deg \xi_i=0$ while $\deg \eta_i=1$ for all $i$.

Note that if we take $\fs_0$ as in Table~\ref{tab3}, then $\fs_0$ becomes a~subalgebra in the Lie superalgebra $\ft_0=\Span(\xi_i\xi_j, \xi_i\eta_j)_{1\leq i, j\leq k}$ preserving the subspace $V=\Span(\xi_i)_{1\leq i\leq k}$. Let us introduce the structure of the cotangent bundle on the space $\Span(\xi_i, \eta_j)_{1\leq i, j\leq k}$ having identified $\eta_i$ with $\xi_i^*$. Then, the Cartan prolong $\ft:=(\fg_{-1},\ft_0)_*$ is the Lie superalgebra preserving this structure of the cotangent bundle.
This remark clarifies the geometrical meaning of formulas~\eqref{(*)}.

{\bf (iii) $\boldsymbol{0\neq\fs_{-1}\subsetneq\fg_{-1}}$ and $\boldsymbol{\fg=\fvect(0|n)}$ or $\boldsymbol{\fsvect(0|n)$ or $\widetilde\fsvect(0|n)}$.}
Then, for any solvable subalgebra $\ft_0\subset \fg_0$ preserving $\fs_{-1}$, we can chose a~maximal solvable subalgebra $\fs_0\subset \fg_0$ so that $\ft_0\subset \fs_0$ and $\fs_0$ also preserves $\fs_{-1}$. Therefore, up to a~renumbering of indeterminates, we can assume that for $1<k<n$ we have
\[
\fs_{-1}=\begin{cases}\langle \del_1,\dots, \del_k\rangle&\text{for $\fvect$ or $\fsvect$,}\\
\langle \del_1,\dots, \del_k\rangle\otimes (1+t\Xi)&\text{for $\widetilde\fsvect(0|n)$},
\end{cases}
\]
and $\fs_0$ is as described in Table~\ref{tab3}. Then, we have the following direct sum, each of the summands is a~subalgebra, moreover, the sum of the first and second summands is semi-direct (the second is an ideal), the sum of the first and third summands is semi-direct (the third is an ideal), but the sum of the second and the third summands is only the sum of subspaces (recall the definition~\eqref{ms} of $\fm\fs0^{\fg}$)
\begin{gather}
(\fs_{-1}, \fs_{0})_*^{\fvect(0|n)}= \fm\fs\fc^{\fvect(0|k)}(\xi_1,\dots,\xi_k)\oplus \bigl(\bigoplus_{i>0} \Lambda^i(\xi_{k+1},\dots,\xi_n)\bigr)\otimes \fvect(\xi_1,\dots,\xi_k) \nonumber\\
\hphantom{(\fs_{-1}, \fs_{0})_*^{\fvect(0|n)}=}{}
\oplus\Lambda(\xi_{1},\dots,\xi_k)\otimes \fm\fs0^{\fvect(\xi_{k+1},\, \dots, \, \xi_n)},\nonumber\\
(\fs_{-1}, \fs_{0})_*^{\fsvect(0|n)}= (\fs_{-1}, \fs_{0})_*^{\fvect(0|n)}\cap\fsvect(0|n),\nonumber\\
(\fs_{-1}, \fs_{0})_*^{\widetilde\fsvect(0|2n)}= (\fs_{-1}, \fs_{0})_*^{\fvect(0|2n)}\otimes(1+\Xi)\label{(**)}.
\end{gather}

\begin{Proposition}\label{P3} Let $\fg=\fvect(0|n)$ or $\fg=\fsvect(0|n)$ for $n\geq 3$ or $\widetilde{\fsvect}(0|2n)$ for $n\geq 2$. Let $\fs_{-1}\subsetneq\fg_{-1}$ be any non-zero proper subspace and $\fs_{0}$ the maximal solvable Lie subalgebra in $\fg_0$ preserving $\fs_{-1}$. Then, $(\fs_{-1}, \fs_0)_*^\fg$ is maximal solvable in $\fg$.
\end{Proposition}

\begin{proof}
By Lemmas~\ref{L1} and \ref{L3}, the Lie superalgebra $\fs:=(\fs_{-1}, \fs_0)_*^\fg$ is maximal among all solvable subalgebras in $\fg$ with the given negative component. Hence, the only thing we should check now is that $\fs$ is not contained in any subalgebra $\ft$ of the same type, i.e., maximal among all solvable subalgebras with a~fixed but larger negative component.
Indeed, if $\ft_{-1}=\langle \del_1,\dots,\del_k; \del_{k+1},\dots,\del_{l}\rangle $, then $v:=\xi_1\xi_{k+1}\del_2$ belongs to $\fs$, but not to $\ft$, unless $\fs=\fsvect(0|n)$ and $\dim\fs_{-1}=1$. In the latter case, ${v:=\xi_1\xi_{2}\del_2-\xi_1\xi_{3}\del_3}$ belongs to $\fs$, but not to~$\ft$.
\end{proof}

Let us elucidate formulas \eqref{(**)}. Let $V\subset \fg_{-1}$.
By Lemma \ref{L3}, if $\fs_{-1}=V$, then $\fs\subset V_*^\fg$.
Clearly, the subalgebra $V_*^\fg$ consists of all the vector fields in $\fg$ that preserve $V$, whereas $\fs$ is a~maximal solvable subalgebra in $V_*^\fg$. In the case where $\fg=\fvect(0|n)$, the form of elements of~$V_*^\fg$ is absolutely clear. Let ${V:=\Span(\partial_1, \dots, \partial_k)}$, and $\fvect_{\geq 0}:=\bigoplus_{i\ge 0} \fvect_i$. Then, we have the direct sum of subspaces (compare with formulas~\eqref{(**)})
\[
V_*^{\fvect(0|n)}= \Lambda(\xi_{k+1},\dots, \xi_n)\otimes \fvect(\xi_1,\dots, \xi_k)\oplus \Lambda(\xi_1,\dots,\xi_k)\otimes \fvect_{\geq 0}(\xi_{k+1},\dots, \xi_n).
\]

{\bf (iii) $\boldsymbol{0\neq V=\fs_{-1}\subsetneq\fg_{-1}}$ where $\boldsymbol{\fg=\fh(0|n)}$ for $\boldsymbol{n>4}$.} As we have observed in Lem\-ma~\ref{L3}, in this case, the subalgebra $\fs$ is contained in the maximal subalgebra $V_*\subset \fg$ with the given negative component $V$ and $\fs$ is a maximal solvable subalgebra in $V_*$. In particular, $\fs_0$ is a maximal solvable subalgebra in the stabiliser $\text{St}(V)\subset \fg_0$ of $V$.

Let $B$ be the non-degenerate symmetric and $\fg_0=\fo(n)$-invariant bilinear form on $\fg_{-1}$ corresponding to the symplectic form $\omega$. Since $\text{St}(V)$, and therefore its subalgebra $\fs_0$, preserve~$V$, they also preserve the subspaces $V^\perp$, $V\cap V^\perp= {\rm{Ker}} B|_V$ and $V+V^\perp$. This means that, in general, the subalgebra $\fs_0$, being a maximal solvable subalgebra in $\text{St}(V)$, might be not maximal solvable subalgebra in the whole component~$\fg_0$.

Introduce the following basis in $\fg_{-1}$:
\begin{gather}
\text{$\xi=(\xi_1, \dots , \xi_k)$ a basis in $V\cap V^\perp={\rm{Ker}} B|_V$,}\nonumber\\
\text{$\eta=(\eta_1, \dots, \eta_k)$ the dual to $\xi$ basis in the complement to $V+V^\perp$ in $\fg_{-1}$,}\nonumber\\
\text{$\alpha=\bigl(\xi_1^a, \dots, \xi_l^a, \eta^a_1,\dots, \eta_l^a, \zeta^a\bigr)$ a basis in the complement to $V\cap V^\perp$ in $V$,}\nonumber\\
\text{$\beta=\bigl(\xi_1^b, \dots, \xi_m^b, \eta^b_1,\dots, \eta_m^b, \zeta^b\bigr)$ a basis in the complement to $V\cap V^\perp$ in $V^\perp$.}\label{klm}
\end{gather}
We will order the basis elements so: $\xi$, $\eta$, followed by $\alpha$, and then $\beta$. Observe that all solvable subalgebras of $\fg_0$ are contained in the maximal solvable subalgebra from Table~\ref{tab3}.

In certain cases, not all these basis elements are present. For example, if $V\subset V^\perp$, i.e., the restriction of the form $B$ to the subspace $V$ is zero, there are no elements $\alpha$. The other way round, if $V^\perp\subset V$, then there are no elements $\beta$. Finally, if $V^\perp=V$, then the subspace $V$ is Lagrangian and only elements $\xi, \eta$ remain. If $V\cap V^\perp=0$, i.e., the restriction of $B$ to $V$ is non-degenerate, then there are no elements $\xi$, $\eta$.
Besides, if the codimension of $V\cap V^\perp$ in $V$ (resp.\ in~$V^\perp$) is even, then the element $\zeta^a$ (resp.\ $\zeta^b$) is absent. Nevertheless, in order not to overburden the text by considering various cases we will use all the basis elements assuming, when needed, that some of them are absent, i.e., are equal to~0.

We say that the subspace $V\subset \fg_{-1}$ is \textit{singular} if $\dim \Ker B|_V\leq 1$ and $\codim \Ker B|_V=1$ in~$V^\perp$.
In coordinates, this means that of all basis elements of type $\beta$ there is only one of $\zeta^b$ and either no elements of type $\xi$, $\eta$ at all, or there is just one of type $\xi$ and one of type $\eta$.

\begin{Proposition}\label{ThH} Let $\fg=\fh(0|n)$ for $n>4$. If the subspace $V\subset \fg_{-1}$ is non-singular, then the subalgebra $\fm\fs V^{\fh(0|n)}$ is maximal solvable in $\fh(0|n)$.
\end{Proposition}

\begin{proof} First, let us describe the Lie superalgebra $V_*$. First of all, note that among the elements of the subspace $V$ acting on functions generating $\fh(0|n)$, there are derivations with respect to all indeterminates of types $\eta$ and $\alpha$, but there are no derivations with respect to indeterminates of types $\xi$ and $\beta$. Therefore, the Lie superalgebra $V_*$ contains a subspace $W$:
\[
W=\Span(f\in \fh(0|n)\mid f   \ \text{is a monomial and} \ \deg_{\xi,\beta}f\ge 2).
\]
Moreover, the subspace $W$ is an ideal in $V_*$ and is contained in $(V_*)_{\ge 0}$. The ideal $W$ is the direct sum of the three subspaces $W=W^{2,0}\oplus W^{1,1}\oplus W^{0,2}$, where
\begin{gather*}
W^{2,0}=\Span(f\in W\mid f\   \text{is a monomial and}\   \deg_\xi f\ge 2,\, \deg_\beta f=0),\\
W^{1,1}=\Span(f\in W\mid f \   \text{is a monomial and}\  \deg_\xi f\ge 1, \,\deg_\beta f\ge 1),\\
W^{0,2}=\Span(f\in W\mid f\   \text{is a monomial and}\   \deg_\xi f=0,\, \deg_\beta f\ge 2).
\end{gather*}

Observe that $W^{2, 0}$, $W^{0,2}$ and $W^{2,0}+W^{1,1}$ are subalgebras, but not ideals, in $W$.

The subalgebra $W^{0,2}$ is isomorphic to the tensor product $\fh(\beta)_{\ge 0}\otimes \fpo(\alpha) \otimes \Lambda(\eta)$, understood as the product of Poisson superalgebras, see \cite{Z}, i.e., on each of the factors there are two operations: a supercommutative and associative multiplication and the Poisson bracket (the zero one on the third factor) and these two multiplications are related by means of the Leibniz rule, so the Poisson bracket on the product is given by the formula{\samepage
\begin{gather*}
[f_1(\beta)\otimes g_1(\alpha)\otimes h_1(\eta),   f_2(\beta)\otimes g_2(\alpha)\otimes h_2(\eta)]\\
=
\pm [f_1(\beta),f_2(\beta)]\otimes g_1(\alpha)g_2(\alpha)\otimes h_1(\eta)h_2(\eta) \pm f_1(\beta)f_2(\beta)\otimes[g_1(\alpha), g_2(\alpha)]\otimes h_1(\eta)h_2(\eta),
\end{gather*}
where the signs $\pm$ are governed by the sign rule.}

Therefore, the solvable subalgebra $\fm\fs V^{\fh(0|n)}$ contains the product $\fm\fs 0^{\fh(\beta)}\otimes \fpo(\alpha)\otimes \Lambda(\eta)$ coming from the subalgebra $W^{0,2}$.

The complement to the ideal $W$ in $V_*$ is the direct sum of two subspaces $W^{1,0}\oplus W^{0,0}$, where
\begin{gather*}
 W^{1,0}=\Span(f\in W\mid f   \ \text{is a monomial and}\ \deg_\xi f\ge 1,\, \deg_\beta f=0) \\
 \hphantom{W^{1,0}}{} =\Span(\psi(\alpha)\sum_i \xi_i\varphi_i(\eta)),\\
 W^{0,0}=\Span(f\in W\mid f \   \text{is a monomial and}  \ \deg_\xi f=0
= \deg_\beta f)
=\Span(f(\alpha))\simeq \fh(\alpha).
\end{gather*}

Observe that $\Span (\sum_i \xi_i\varphi_i(\eta))\simeq \fvect(\eta)$. The subspace $W^{1,0}$ is not a subalgebra. A subalgebra containing $W^{1,0}$ is a semi-direct sum $W^{1,0}\niplus W^{2,0}$, and
$(W^{1,0}\niplus W^{2,0})/W^{2,0} \simeq \fvect(\eta)\otimes \Lambda(\alpha)$. Therefore, we can represent $W^{1,0}\niplus W^{2,0}$ as $\fvect(\eta)\otimes \fpo(\alpha)\niplus W^{2,0}$, so its solvable part is
\[
\fm\fs\fc^{\fvect(\eta)}\otimes \fpo(\alpha)\niplus W^{2,0}.
\]
Finally, each of the subspaces $W^{2,0}$, $W^{1,1}$, $W^{0,2}$ and $W^{1,0}\niplus W^{2,0}$ is invariant under
$W^{0,0}\simeq\fh(\alpha)$.
Thus,
\[
V_*=\fh(\alpha)\inplus(\fvect(\eta)\otimes \fpo(\alpha)+W^{2,0}+W^{1,1}+\fh(\beta)_{\ge 0}\otimes \fpo(\alpha) \otimes \Lambda(\eta)),
\]
and the maximal solvable subalgebra $\fs=\fm\fs V^{\fh(0|n)}$ with the negative component $\fs_{-1}=V$ is of the form
\begin{gather}
\fs=\fm\fs V^{\fh(0|n)}=\fm\fs\fc^{\fh(\alpha)}\inplus (\fm\fs\fc^{\fvect(\eta)}\otimes \fpo(\alpha)+W^{2,0}+W^{1,1}\nonumber\\
\hphantom{\fs=}{}+\fm\fs 0^{\fh(\beta)}\otimes \fpo(\alpha) \otimes \Lambda(\eta)).\label{msV}
\end{gather}

This $\fs$ is the maximal solvable subalgebra among all solvable subalgebras with a given negative part. To see that $\fs$ is maximal among all solvable subalgebras, it suffices to verify that if $\ft=(\ft_{-1},\ft_0)_*$ and $\ft_{-1}\supset\fs_{-1}$ whereas $\ft_{0}\supset\fs_{0}$ is solvable, then $\fs$ is not contained in $\ft$.

First, note that the subspace $\ft_{-1}$ is necessarily invariant with respect to $\ft_0$, and hence with respect to $\fs_0$. Therefore, we are left with the three cases:
\begin{enumerate}\itemsep=0pt
\item[(i)] if $m\ne 0$, then the space $\ft_{-1}$ must contain $\xi^b_1$;
\item[(ii)] if $m=0$, but $V^\perp$ is not contained in $V$, then the space $\ft_{-1}$ must contain $\zeta^b$;
\item[(iii)] if $V^\perp\subset V$, then the space $\ft_{-1}$ must contain $\eta_k$.
\end{enumerate}

In each of these cases, we indicate an element $u\in \fs_1$ which is contained in the subalgebra $\fs$, but can not be contained in $\ft$. In the majority of these cases, $u\in W$. Each case contains also subcases occasioned by a low dimension of the subspaces involved; some of these subcases are exceptional because the subspace $V$ is singular.

In case (i) for $k\ge 1$, we can take $u=\xi_1^b\eta_1^b \eta_1\in W^{0,2}$. Then, $[\xi_1^b, u]=-\xi_1^b \eta_1\notin \ft_0$, because the basis elements of type $\beta$ go after $\xi$, $\eta$. If $k=0$, i.e., the restriction of the form $B$ to $V$ is non-degenerate, we can take (note that for $l=0$ we have $m\ge 2$, see \eqref{klm}, since $n>4$){\samepage
\[
u=\begin{cases}\xi_1^b\eta_1^b \eta_1^a\in W^{0,2}&\text{if $l\ge 1$},\\
\eta_1^b\eta_2^b\zeta_1^a\in W^{0,2}&\text{if $l=0$,}
\end{cases}\quad \Longmapsto\quad [\xi_1^b, u]=\begin{cases}-\xi_1^b \eta_1^a\notin \ft_0,\\
-\eta_2^b \zeta_1^a\notin \ft_0,\end{cases}
\]
because the elements of type $\beta$ follow the elements of type $\alpha$.}

In case (ii), of all elements of type $\beta$ we have just one $\zeta^b$. For $k\ge 2$, we
 take $u=\xi_2\eta_1\zeta^b\in W^{1,1}$. Then,
$[\zeta^b, u]=\xi_2\eta_1\notin \ft_0$.

The subcases $k=1$ and $k=0$ are exceptional because the subspace $V$ is singular.

For $k=1$, the dimensions of the spaces spanned by each of the elements of type $\xi$ and $\beta$ are equal to 1. Hence, $W^{2,0}=W^{0,2}=0$ and $W=W^{1,1}$. The subspace $W^{1,0}$ is a subalgebra, the sum $\fn= W^{1,0}+W^{1,1}$ is a solvable ideal in the subalgebra $V_*$; this ideal consists of functions ``divisible by~$\xi_1$''. Accordingly, the maximal solvable subalgebra contained in $V_*$ is 
$\fm\fs V^{\fh(0|n)} =\fm\fs\fc^{\fh(\alpha)}\niplus\fn$. Note that this subalgebra is invariant under bracketing with $\zeta^b$, so $\fm\fs V^{\fh(0|n)} \subset\fm\fs \widetilde V^{\fh(0|n)}$, where $\widetilde V=V\oplus \langle\zeta^b\rangle$.

For $k=0$, there are no elements $\xi$, $\eta$, the codimension of $V$ in $\fg_{-1}$ is equal to 1, and the restriction of the form $B$ to $V$ is non-degenerate. Therefore, the ideal in formula \eqref{msV} vanishes, and hence $\fm\fs V^{\fh(0|n)}= \fm\fs\fc^{\fh(\alpha)}$. Since $\fh(\alpha)$ commutes with $\zeta^b$, we see that in this case, $\fm\fs V^{\fh(0|n)}\subset \fm\fs\fc^{\fh(0|n)}$.

In case (iii), there are no elements of type $\beta$ and, if the elements of type $\alpha$ are present, i.e., when $V^\perp \ne V$, we can take
\[
u=\begin{cases}\xi_k\eta_k\eta_1^a&\text{for $l\ge 1$},\\
\xi_k\eta_k\zeta^a&\text{for $l=0$,}\end{cases}\quad \Longmapsto\quad  [\eta_k,u]=\begin{cases}\eta_k\eta_1^a\notin \ft_0,\\
\eta_k\zeta^a\notin \ft_0.\end{cases}
\]

If $V^\perp = V$, there remain only elements $\xi$, $\eta$; moreover, $k\ge 3$ because $n>4$. In this case, we can take $u=\xi_k\xi_2\eta_1$. Then, $[\eta_k,u]=\xi_2\eta_1\notin \ft_0$. \end{proof}

\section[The two remaining cases of small dimension: vect(0|2) and h'(0|4)]{The two remaining cases of small dimension:\\ $\boldsymbol{\fvect(0|2)}$ and $\boldsymbol{\fh'(0|4)}$}\label{ss5}

For completeness, let us consider also the cases of $\fvect(0|2)\simeq \fsl(1|2)$ and $\fh'(0|4)\simeq \fp\fsl(2|2)$ whose maximal solvable subalgebras were described in their supermatrix realisation in~\cite{S2}.

\begin{Proposition}\label{P5} The maximal graded solvable subsuperalgebras in $\fg=\fvect(0|2)$ are listed in Table~{\rm \ref{vect02}}.
\begin{table}[h]\centering \renewcommand{\arraystretch}{1.2}
\begin{tabular}{|c|c|c|c|}
\hline
name of $\fs$ & basis of $\fs_{-1}$ & basis of $\fs_0$ & basis of $\fs_1$\\
\hline
$\fm\fs V$ & $\partial_1 $ & $\xi_1\partial_1$, $\xi_2\partial_1$, $\xi_2\partial_2$ & $\xi_1\xi_2\partial_1$, $\xi_1\xi_2\partial_2$ \\
\hline
$\fm\fs\fc$ & $\partial_1$, $\partial_2 $ & $\xi_1\partial_1$, $\xi_2\partial_1$, $\xi_2\partial_2$ & $\xi_1\xi_2\partial_1$ \\
\hline
\end{tabular}
\caption{}\label{vect02}
\end{table}

The maximal graded solvable subsuperalgebras in $\fg=\fh'(0|4)$ are listed in Table~{\rm \ref{h04}}.
\begin{table}[h]\centering \renewcommand{\arraystretch}{1.2}
\begin{tabular}{|c|c|c|c|}
\hline
name of $\fs$ & basis of $\fs_{-1}$ & basis of $\fs_0$ & basis of $\fs_1$\\
\hline
$\fm\fs V$ & $\xi_1 $ & $\xi_1\eta_1$, $\xi_1\xi_2$, $\xi_1\eta_2$, $\xi_2\eta_2$ & $\xi_1\eta_1\eta_2$, $\xi_1\xi_2\eta_1$, $\xi_1\xi_2\eta_2$, $\xi_2\eta_1\eta_2$ \\
\hline
$\fm\fs\fc$ & $\xi_1$, $\xi_2,\eta_1,\eta_2 $ & $\xi_1\eta_1$, $\xi_1\xi_2$, $\xi_1\eta_2$, $\xi_2\eta_2$ & $\xi_1\xi_2\eta_2$ \\
\hline
$\fm\fs\widetilde V$ & $\xi_1$, $\xi_2,\eta_2 $ & $\xi_1\eta_1$, $\xi_1\xi_2$, $\xi_1\eta_2$, $\xi_2\eta_2$ & $\xi_1\eta_1\eta_2$, $\xi_1\xi_2\eta_1$, $\xi_1\xi_2\eta_2$ \\
\hline
\end{tabular}
\caption{}\label{h04}
\end{table}
\end{Proposition}

\begin{proof} If $\fg=\fvect(0|2)$, then the subalgebra $\fm\fs 0^{\fvect(0|2)}$ is invariant under the action of $\partial_1$, and therefore is contained in $\fm\fs V$, where $V=\langle \partial_1\rangle$, and is not maximal. Thus, we get precisely two $\Zee$-graded maximal subalgebras, see Table~\ref{vect02}. Obviously, although these subalgebras are not isomorphic as $\Zee$-graded subalgebra, they are isomorphic as abstract subalgebras. This corresponds with the description of maximal graded solvable subsuperalgebras of $\fsl(1|2)$ obtained in~\cite{S2}.

Now, consider $\fg=\fh'(0|4)$. As stated in~\cite{S2}, both the maximal solvable subalgebra $\fs$ of $\fsl(2|2)$, and of its quotient $\fp\fsl(2|2)$ modulo center, are described by the collection of superdimensions of irreducible quotients under the tautological action of $\fs$ in the $2|2$-dimensional superspace. For the $2|2$-dimensional superspace, there are only three such collections: $\{(1|0), (1|1), (0|1)\}$,\, $\{(1|1), (1|1)\}$ and $\{(2|2)\}$. Note that the last collection corresponds to a 1-parameter family of subalgebras conjugate in $\fgl(2|2)$ (resp.\ in $\fp\fgl(2|2)$), but not in $\fsl(2|2)$ (resp.\ not in $\fp\fsl(2|2)$).\looseness=1

Now, let us look how the constructions, used in this note in other cases, work in the description of maximal graded solvable Lie superalgebras in $\fh'(0|4)$. As in the case of $\fvect(0|2)$, the subalgebra $\fm\fs 0$ is invariant under the action of $\xi_1$, and therefore is contained in $\fm\fs V$, where $V=\langle \xi_1\rangle$, and hence is not maximal. This subalgebra, as well as the isomorphic to it as abstract, but not as $\Zee$-graded, subalgebra $\fm\fs\fc$ corresponds to the collection of superdimensions $\{(1|0), (1|1), (0|1)\}$. To the subalgebra $\fm\fs\widetilde V$, where $\widetilde V=\Span(\xi_1,\xi^a,\eta^a)$, there corresponds the collection $\{(1|1), (1|1)\}$. As a result, we get three maximal $\Zee$-graded solvable subalgebras, two of which are isomorphic as abstract ones, see Table~\ref{h04}.
\begin{table}[h]\centering\renewcommand{\arraystretch}{1.2}
\begin{tabular}{|c|c|c|c|}
\hline
name of $\fs$ & basis of $\fs_{-1}$ & basis of $\fs_0$ & basis of $\fs_1$\\
\hline
$\fm\fs V$ & $\xi_1 $ & $\xi_1\eta_1, \xi_1\xi_2, \xi_1\eta_2, \xi_2\eta_2$ & $\xi_1\eta_1\eta_2, \xi_1\xi_2\eta_1, \xi_1\xi_2\eta_2, \xi_2\eta_1\eta_2$ \\
\hline
$\fm\fs\fc$ & $\xi_1, \xi_2,\eta_1,\eta_2 $ & $\xi_1\eta_1, \xi_1\xi_2, \xi_1\eta_2, \xi_2\eta_2$ & $\xi_1\xi_2\eta_2$ \\
\hline
$\fm\fs\widetilde V$ & $\xi_1, \xi^a, \eta^a$ & $\xi_1\eta_1, \xi_1\xi^a, \xi_1\eta^a, \xi^a\eta^a$ & $\xi_1\eta_1\eta^a, \xi_1\eta_1\xi^a, \xi_1\xi^a\eta^a$ \\
\hline
\end{tabular}
\caption{}\label{ill}
\end{table}

Note that the 1-parameter family of maximal solvable subalgebras in $\fsl(2|2)$ (and in $\fp\fsl(2|2)$) corresponding to the collection $\{(2|2)\}$ consists of \textit{non-graded} subalgebras, and therefore this family is absent in this note devoted to the classification of
\textit{graded} subalgebras.

Observe also that five subalgebras of the form $\fm\fs V$ are not maximal. If $\dim V=3$, and the restriction of the form $B$ to the subspace $V$ is non-degenerate, and also if $\dim V=2$ and $\dim \Ker B|_V=1$, then $V$ is singular. The other three cases are occasioned by very small dimension of all subspaces involved.
\begin{table}[h]\centering\renewcommand{\arraystretch}{1.2}
\begin{tabular}{|c|c|c|c|}
\hline
name of $\fs$ & basis of $\fs_{-1}$ & basis of $\fs_0$ & basis of $\fs_1$\\
\hline
$\fm\fs V$ & $\zeta $ & $\xi^b\eta^b$, $\xi^b\zeta^b$ & $\xi^b\eta^b\zeta^b$, $\xi^b\zeta^b\zeta$, $\xi^b\eta^b\zeta$ \\
\hline
$\fm\fs \widetilde V$ & $\zeta$, $\xi^b$, $\zeta^b $ & $\xi^b\eta^b$, $\xi^b\zeta^b$, $\xi^b\zeta$, $\zeta\zeta^b$ & $\xi^b\eta^b\zeta^b$, $\xi^b\zeta^b\zeta$, $\xi^b\eta^b\zeta$ \\
\hline\hline
$\fm\fs V$ & $\xi_1$, $\xi_2$ & $\xi_1\eta_1$, $\xi_1\xi_2$, $\xi_1\eta_2$, $\xi_2\eta_2,$ & $\xi_1\eta_1\eta_2$, $\xi_1\xi_2\eta_1$, $\xi_1\xi_2\eta_2$ \\
\hline
$\fm\fs\widetilde V$ & $\xi_1$, $\xi_2,\eta_2 $ & $\xi_1\eta_1$, $\xi_1\xi_2$, $\xi_1\eta_2$, $\xi_2\eta_2$ & $\xi_1\eta_1\eta_2$, $\xi_1\xi_2\eta_1$, $\xi_1\xi_2\eta_2$ \\
\hline\hline
$\fm\fs V$ & $\xi^a$, $\eta^a$ & $\xi^a\eta^a$, $\xi^b\eta^b$ & $\xi^b\eta^b\xi^a$, $\xi^b\eta^b\eta^a$ \\
\hline
$\fm\fs\widetilde V$ & $\xi^a$, $\eta^a,\xi^b$ & $\xi^a\eta^a$, $\xi^b\eta^b$, $\xi^a\xi^b$, $\xi^b\eta^a$ & $\xi^b\eta^b\xi^a$, $\xi^b\eta^b\eta^a$, $\xi^b\xi^a\eta^a$ \\
\hline\hline
\end{tabular}
\caption{}\label{5sub}
\end{table}

Table~\ref{5sub} describes these three subalgebras
$\fm\fs V$ and their ambients $\fm\fs \widetilde V$.
\end{proof}

\subsection*{Acknowledgements}
I am thankful to A.L.~Oni\-shchik who raised the problem, and D.~Leites for help. A~part of the results of this note was announced in \cite{S1}; their proofs were sketched in Report no.~32/1988-15 of Department of Mathematics of Stockholm University to which I am thankful for hospitality.

\pdfbookmark[1]{References}{ref}
\LastPageEnding

\end{document}